\documentclass[11pt]{article}

\usepackage{amsmath,amsthm,amssymb}
\usepackage[english]{babel}
\usepackage[T1]{fontenc}
\usepackage{color}
\usepackage{hyperref}

\usepackage[title]{appendix}

\newtheorem{theorem}{Theorem}
\newtheorem{corollary}[theorem]{Corollary}
\newtheorem{lemma}[theorem]{Lemma}

\newtheorem*{remark}{Remark}
\newtheorem*{definition}{Definition}
\newtheorem*{proposition}{Proposition}

\newcommand{\N}{\mathbb{N}}

\newcommand{\R}{\mathbb{R}}
\newcommand{\C}{\mathbb{C}}
\newcommand{\EE}{\mathbb{E}}
\newcommand{\PP}{\mathbb{P}}

\newcommand{\di}{\,\mbox{d}}

\DeclareMathOperator{\Erfc}{Erfc}

\DeclareMathOperator{\Id}{Id}
\DeclareMathOperator{\Res}{Res}
\DeclareMathOperator{\tr}{trace}
\DeclareMathOperator{\dist}{dist}

\def\1{{\mathchoice {1\mskip-4mu\mathrm l}   
{1\mskip-4mu\mathrm l}
{1\mskip-4.5mu\mathrm l} {1\mskip-5mu\mathrm l}}}

\begin{document}
\title{Persistence exponents via perturbation theory:\ AR(1)-processes}
\author{\renewcommand{\thefootnote}{\arabic{footnote}}Frank Aurzada\footnotemark[1]\and \renewcommand{\thefootnote}{\arabic{footnote}}Marvin Kettner\footnotemark[1]}

\footnotetext[1]{Technische Universit\"at Darmstadt, Schlossgartenstra\ss e 7, 64287 Darmstadt, Germany. E-mail: aurzada@mathematik.tu-darmstadt.de, kettner@mathematik.tu-darmstadt.de}

\maketitle

\begin{abstract}
For AR(1)-processes $X_n=\rho X_{n-1}+\xi_n$, $n\in\N$, where $\rho\in\R$ and $(\xi_i)_{i\in\N}$ is an i.i.d.\ sequence of random variables,\ we study the persistence probabilities $\PP(X_0\ge 0,\dots, X_N\ge 0)$ for $N\to\infty$. For a wide class of Markov processes a recent result \cite{Aurzada} shows that these probabilities decrease exponentially fast and that the rate of decay can be identified as an eigenvalue of some integral operator. We discuss a perturbation technique to determine a series expansion of the eigenvalue in the parameter $\rho$ for normally distributed AR(1)-processes.
\end{abstract}

\noindent \textbf{Keywords.} autoregressive process; eigenvalue problem; integral equation; Markov chain; persistence; perturbation theory

\section{Introduction}

The major question in persistence is to understand the behaviour of a
stochastic process in the case it has an unusually long excursion. A first
goal in this context is to compute the rate of decay of the probability
\[\PP(X_0 \geq  0, \ldots, X_N\geq 0),\qquad \text{as $N\to\infty$},\]
where $(X_n)_{n\in\N}$ is a real-valued stochastic process.

Persistence probabilities have received significant attention both
classically and recently. We refer to the surveys
\cite{braymajumdarschehr} (from a theoretical physics point of view) and
\cite{aurzadasimon} (for a mathematics point of view) as well as to the
monographs \cite{metzleretal,redner}.

The guiding idea for the relevance of persistence in the context of
physical systems can be sketched as follows. Consider a complicated
spatial physical system started in some disordered state. When looking at
some specific spatial point, one can ask the question when the state of
this point has changed significantly compared to the initial state. 
The probability of this taking rather long, which is
clearly a persistence probability, is considered to be a measure for the
relaxation time of the system. Even though the system may be very
complicated due to non-trivial interactions, this quantity might still be
accessible, contrary to global quantities.

The present paper deals with discrete time Markov chains on a general
state space. It is well-known that in the case of Markov processes,
non-exit probabilities should have close relations to eigenvalues of
operators. However, establishing such a connection is often non-trivial.
The purpose of the present paper is to establish such a connection for a
specific class of Markov chains and use it to apply results from perturbation
theory for linear operators \cite{Kato}. The setup is as follows.

Let $(\xi_i)_{i\ge 1}$ be a sequence of i.i.d.\ random variables with density $\phi$ and let $\rho\in\R$ be a constant. Moreover, let $X_0$ be a random variable independent of $(\xi_i)_{i\ge 1}$. A one-dimensional autoregressive process is defined by
\[X_n:= \rho X_{n-1}+\xi_n, \quad n\ge 1.\]

The process $(X_n)_{n\in\N}$ is a Markov chain with starting point $X_0$ and transition probabilities $p(x,A)=\int_A \phi(y-\rho x) \di y$.

Although the structure of this process is quite simple, it is difficult to obtain asymptotic results of the persistence probabilities and the exact asymptotic behaviour is still an open problem. Very often these probabilities tend to zero exponentially fast and we are interested in the rate of the decay, the so-called persistence exponent.
Let $P\colon L^\infty(\R)\to L^\infty(\R)$ be defined by $Pf(x):=\int_\R f(y)p(x,\mbox{d} y)$. Furthermore, we set 
\[ P^+\colon L^\infty([0,\infty))\to L^\infty([0,\infty)),\quad P^+(f)(x):= P(f\1_{[0,\infty)})(x)\]
and let $X_0\sim\mu$. Based on the observation

\begin{equation}\label{eq1}
\PP(X_0\ge 0,\dots,X_N\ge 0)=\int_0^\infty (P^+)^N(\1)\di\mu,
\end{equation}

we relate the persistence exponent to an eigenvalue of (a modification of) $P^+$. In addition, we aim to give a series expansion (in the parameter $\rho$) of the desired persistence exponent. Since the proofs depend on a certain transform of the integral operator, we will however restrict our attention to the case where the $(\xi_i)_{i\ge 1}$ are Gaussian.

The outline of this paper is as follows. Section \ref{Section Perron-Frobenius} establishes the relation between the persistence exponent and the largest eigenvalue of some self-adjoint Hilbert-Schmidt integral operator. It is worth pointing out that Lemma \ref{Lemma Perron-Frobenius}, which is stated also in that section, may make it possible to generalize results of this paper to other Markov processes. In Section \ref{Section Peturbation AR Gaussian}, our main result, the series expansion of the persistence exponent, is stated. Section \ref{Section related work} contains related work, possible generalizations, and further remarks. Section \ref{Section Proofs} is devoted to the proofs. In Appendix \ref{appendix} we give a brief exposition of auxiliary results from perturbation theory.\\


\section{Results}

\subsection{Connection between persistence probabilities and an \newline eigenvalue problem} \label{Section Perron-Frobenius}
Unless otherwise stated we assume that $\xi_1$ and $X_0$ are standard normally distributed.
The canonical integral operator $P^+$ is not suitable to relate the persistence exponent to an eigenvalue, due to compactness problems, i.e.\ for $\rho>0$ and any $n\ge 1$, the integral operator $(P^+)^n$ is not compact \cite[Remark 2.13]{Baumgarten2}. For this reason, we consider a modification of the canonical integral operator which satisfies a certain compactness and irreducibility condition and allows to establish the connection between the persistence exponent and an eigenvalue. Moreover, this operator is very suitable for perturbation theory.

\begin{definition}\label{Defintion of M_rho}
Let $\gamma$ be the standard Gaussian measure on $\R$, i.e. $\di\gamma(x)= \frac{1}{\sqrt{2\pi}}e^{-\frac{x^2}{2}}\di x$. For $-1<\rho<1$, let $M_\rho$ be given by
\[M_\rho\colon L^2([0,\infty),\gamma)\to L^2([0,\infty),\gamma),\quad M_\rho f(x)= \int_0^\infty f(y)m_\rho(x,y)\di \gamma(y),\]
where $m_\rho(x,y):= \frac{1}{\sqrt{1-\rho^2}}e^{-\frac{\rho^2x^2+\rho^2y^2-2\rho xy}{2(1-\rho^2)}}$.
\end{definition}

The operator $M_\rho$ is well-defined, self-adjoint and compact, which is proved in Section \ref{Section Proofs}. Moreover, by the Mehler formula \cite{Mehler}, \cite[Section 4.2]{Janson}, we have
\[m_\rho(x,y) = \sum_{n=0}^\infty \frac{1}{n!}h_n(x)h_n(y)\rho^n.\]
Here, $h_n$ denotes the $n$-th Hermite polynomial given by the formula $h_n(x):= (-1)^n e^{\frac{x^2}{2}}\frac{\di^n}{\di x^n}e^{-\frac{x^2}{2}}$, i.e. the first Hermite polynomials are $h_0(x)=1$, $h_1(x)=x$, $h_2(x)=x^2-1$.

We can now formulate the connection between the persistence probabilities and the eigenvalue problem of $M_\rho$.

\begin{theorem}\label{PE=Largest Eigenvalue}
Let $-1<\rho<1$. Then
\[c_{\rho}\lambda_\rho^N \le \PP(X_0\ge 0,\dots, X_N\ge 0) \le C_\rho \lambda_\rho^N,\]
where $\lambda_\rho\in(0,1)$ is the largest eigenvalue of $M_\rho$ and $c_\rho,C_\rho>0$.
\end{theorem}

\subsection{Perturbation theory for the persistence exponent}\label{Section Peturbation AR Gaussian}

In the following theorem, we prove the existence of a power series for the desired persistence exponent.

\begin{theorem}\label{Largest Eigenvalue holomorphic}
For $\lambda_\rho$ from Theorem \ref{PE=Largest Eigenvalue} we have
\[\lambda_\rho = \sum_{n=0}^\infty K_n\rho^n,\quad K_n\in\R,\]
for $\rho\in (-r_0,r_0)$, where
\[r_0\ge \frac{1}{3}.\]
\end{theorem}

\begin{remark}
Based on numerical calculations, we expect that the radius of convergence is significantly larger than the value $\frac{1}{3}$ that we can prove analytically. It remains an interesting open problem to determine the radius of convergence. \\
However, the radius of convergence can be at most $1$. It was proved e.g.\ in \cite{Baumgarten} that for $\rho \geq 1$ the persistence probability tends to a constant so that one has $\lambda_\rho=1$ for any $\rho\geq 1$.\\
One may further ask whether a power series for the persistence exponent can be obtained if $\rho \le -1$. By \cite{Dembo}, the behaviour of the persistence probabilities is also exponential. It would be very interesting to find any further information about the persistence exponent there.
\end{remark}

As stated above, the kernel $m_\rho$ can be expressed as a power series in $\rho$. Theorem \ref{Largest Eigenvalue holomorphic} shows that the largest eigenvalue can also be expressed as a power series in $\rho$. In addition, the corresponding eigenfunction can be expressed as a power series in $\rho$ (Appendix \ref{appendix}, Theorem \ref{Thm: Eigenvalue, eigenvector holomorphic}). Hence, the eigenvalue equation
\[\lambda_\rho f_\rho = \int_0^\infty f_\rho(y)m_\rho(x,y) \di\gamma (y)\]
can be rewritten as
\[\sum_{n=0}^\infty \rho^n K_n \sum_{i=0}^\infty \rho^i g_i(x) = \int_0^\infty \sum_{i=0}^\infty \rho^i g_i(y) \sum_{n=0}^\infty \rho^n a_n(x,y) \di y\]
with $a_n(x,y)= \frac{1}{n!}h_n(x)h_n(y)$. Because in our case the quantity $a_n(x,y)$ has a ``nice'' product form, a comparison of the coefficients in $\rho$ and $x$ yields that $g_j(x)=\sum_{i=0}^j G_{i,j}h_i(x)$ for some constants $G_{i,j}\in\R$ and we get the following equations
\begin{equation}\label{eigenvalueequation}
\sum_{i=j}^k K_{k-i}G_{j,i} = \sum_{i=0}^{k-j}G_{i,k-j}\psi_{i,j},\quad 0\le j\le k,\quad k\in\N,
\end{equation}

where
\[\psi_{i,j} = \frac{1}{j!} \int_0^\infty h_i(y)h_j(y) \di\gamma(y) ,\quad i,j\ge 0.\]
With this iterative formula we are able to compute the coefficients $(K_n)_n$ explicitly. The first coefficients are given by
\begin{align*}
K_0 &=\frac{1}{2},\\
K_1 &= \frac{1}{\pi},\\
K_2 &= \frac{1}{\pi} - \frac{2}{\pi^2},\\
K_3 &= \frac{7}{6\pi} - \frac{6}{\pi^2} + \frac{8}{\pi^3},\\
K_4 &= \frac{1}{\pi} - \frac{35}{3 \pi^2} + \frac{40}{\pi^3} - \frac{40}{\pi^4},\\
K_5 &= \frac{43}{40\pi} - \frac{19}{\pi^2} + \frac{116}{\pi^3} - \frac{280}{\pi^4} + \frac{224}{\pi^5},\\
K_6 &= \frac{7}{6\pi} - \frac{5149}{180\pi^2} + \frac{790}{3\pi^3} - \frac{3260}{3\pi^4} + \frac{2016}{\pi^5} - \frac{1344}{\pi^6},\\
K_7 &= \frac{117}{112\pi} -\frac{799}{20\pi^2} + \frac{7762}{15\pi^3} -\frac{3164}{\pi^4} + \frac{29456}{3\pi^5} - \frac{14784}{\pi^6} + \frac{8448}{\pi^7},\\
K_8 &= \frac{1}{\pi} - \frac{8843}{168\pi^2} + \frac{16541}{18\pi^3} - \frac{23147}{3\pi^4} + \frac{34944}{\pi^5} - \frac{86688}{\pi^6} + \frac{109824}{\pi^7} - \frac{54912}{\pi^8}.
\end{align*}
Unfortunately, we could not find a helpful structure to obtain a closed-form expression for the $n$-th coefficient.

\subsection{Related work and discussion}\label{Section related work}

The idea of relating the persistence exponent to an eigenvalue of an integral operator has already received great attention (see \cite{Baumgarten2,Aurzada,Wachtel,Majumdarbray}, also see \cite{Champagnat,Collet,Meleard,Tweedie2,Tweedie1} for the quasi-stationary approach). Among these, \cite{Wachtel} is a very recent work, which uses functional analytical methods, such as the Fredholm alternative to obtain a precise asymptotic behaviour of the persistence probabilities. However, the persistence exponent is given there only implicitly. For a completely different approach to persistence probabilities of autoregressive processes via their generating polynomials we refer the reader to \cite{Dembo}. Nevertheless, quantitative statements about the persistence exponent are known only in a few examples (see \cite{Aurzada}).\\
We believe that our techniques and results have a greater generality. In particular, we conjecture that Theorem \ref{PE=Largest Eigenvalue}, Theorem \ref{Largest Eigenvalue holomorphic}, and equations of type \eqref{eigenvalueequation} are valid more generally, e.g.\ for other innovation distributions and also for moving average processes (see e.g. \cite{Aurzada,Majumdardhar}). For this, suitable transformations of the integral operator $P^+$ must be found. Another possible generalization is to use these methods for two-sided exit problems.\\
 
\begin{remark}
Firstly, note that $(X_0,X_1,X_2,\ldots)^T=S^{-1} (X_0,\xi_1,\xi_2,\ldots)^T$, where \[
S=\begin{pmatrix}
1 & 0 & 0 & \ldots & 0 \\
-\rho & 1 & 0 & \ldots  & 0\\
0 & -\rho & 1 &   &0\\
\vdots & & \ddots & \ddots & \vdots\\
0 & \ldots &  0 & -\rho & 1
\end{pmatrix}.
\]

Secondly, observe that since the innovation vector $(X_0,\xi_1,\xi_2,\ldots)$ is i.i.d.\ Gaussian, it is in particular isotropic, so that $U:=\frac{(X_0,\xi_1,\ldots,\xi_N)}{||(X_0,\xi_1,\ldots,\xi_N)||}$ is uniformly distributed on the unit sphere of $\mathbb R^{N+1}$. Therefore, the persistence probability,
$$
\PP(X_0\ge 0,\dots,X_N\ge 0)=\PP(S^{-1} (X_0,\xi_1,\ldots,\xi_N)\in\mathbb R_{\ge 0}^{N+1}) = \PP(U \in S\mathbb R_{\ge 0}^{N+1}),
$$
is the same as the normalized area of the intersection of the unit sphere of $\mathbb R^{N+1}$ with the cone $S\mathbb R_{\ge 0}^{N+1}$.

Our results thus carry over to any isotropic vector $(X_0,\xi_1,\xi_2,\ldots)$. However, this vector generates an AR(1)-process (i.e.\ the innovations are i.i.d.) if and only if the innovations are Gaussian \cite{Letac}.
\end{remark}


\section{Proofs}\label{Section Proofs}
\subsection{\texorpdfstring{Proofs of the results of Section \ref{Section Perron-Frobenius}}{Proofs of the results of Section 2.1}}

We begin by proving the properties of the operator $M_\rho$. Since $m_\rho(x,y) = \frac{1}{\sqrt{1-\rho^2}}e^{-\frac{\rho^2x^2+\rho^2y^2-2\rho xy}{2(1-\rho^2)}}$, we obtain

\begin{align*}
&\int_0^\infty \int_0^\infty m_\rho(x,y)^2 \di\gamma(y)\di\gamma(x)\\
&= \frac{1}{1-\rho^2} \int_0^\infty \int_0^\infty e^{\frac{-\rho^2y^2 - \rho^2x^2 +2\rho xy}{1-\rho^2}} \di\gamma(y)\di\gamma(x)\\
&= \frac{1}{2\pi(1-\rho^2)} \int_0^\infty e^{-\frac{1+\rho^2}{2(1-\rho^2)}x^2} \int_0^\infty e^{-\frac{1+\rho^2}{2(1-\rho^2)}y^2}e^{\frac{2\rho xy}{1-\rho^2}}\di y \di x\\
&=  \frac{1}{2\pi(1-\rho^2)} \int_0^\infty e^{-\frac{1+\rho^2}{2(1-\rho^2)}x^2} \frac{\sqrt{\pi}\sqrt{2(1-\rho^2)}}{2 \sqrt{(1+\rho^2)}}\\
&\qquad\qquad\qquad  \cdot\Erfc\left(\frac{-\sqrt{2}\rho x}{\sqrt{(1-\rho^2)(1+\rho^2)}}\right) e^{\frac{2\rho^2x^2}{(1-\rho^2)(1+\rho^2)}}\di x\\
&\le C \cdot \int_0^\infty e^{\frac{-(1+\rho^2)^2x^2+4\rho^2x^2}{2(1-\rho^2)(1+\rho^2)}} \di x\\
&= C \cdot \int_0^\infty e^{-\frac{(1-\rho^2)}{2(1+\rho^2)}x^2} \di x\\
&<\infty.
\end{align*}

So $m_\rho(\cdot,\cdot)\in L^2(\gamma\otimes\gamma)$ and hence, the operator $M_\rho$ is a Hilbert-Schmidt integral operator \cite[Chapter IV, Proposition 6.5]{Schaefer} and thus well-defined and compact. In addition, the operator is obviously self-adjoint.
\medskip

In preparation for the proof of Theorem \ref{PE=Largest Eigenvalue}, we begin by relating the integral operator $M_\rho$ to the persistence problem of the AR(1)-process.

Let $\tilde{X}_0:= c\cdot X_0$ and $\tilde{\xi}_i:= c\cdot\xi_i$, i.e. $\tilde{\xi}_i\sim\mathcal{N}(0,c^2)$ for a constant $c>0$. Then trivially we have $\PP(X_0\ge 0,\dots,X_N\ge 0) = \PP(\tilde{X}_0\ge 0,\dots,\tilde{X}_N\ge 0)$ for all $N\in\N$. The transition operator $\tilde{P}$ of the Markov chain $(\tilde{X}_n)_n$ is given by the kernel
\[\tilde{p}(x,\mbox{d} y) = \frac{1}{\sqrt{2\pi}\sqrt{c^2}}e^{-\frac{y^2+\rho^2x^2-2\rho xy}{2c^2}}\di y.\]
Write $\tilde{P}^+\colon L^\infty([0,\infty))\to L^\infty([0,\infty)),\ \tilde{P}^+(f)(x)= \tilde{P}(f\1_{[0,\infty)})(x)$  as in the introduction.
For $c=\sqrt{1-\rho^2}$ we obtain
\[\tilde{p}(x,\mbox{d} y) = \frac{1}{\sqrt{2\pi}\sqrt{1-\rho^2}}e^{-\frac{\rho^2x^2+\rho^2y^2-2\rho xy}{2(1-\rho^2)}}e^{-\frac{y^2}{2}}\di y = m_\rho(x,y)\di\gamma(y).\]
This and the fact that $L^2(\gamma)$ contains all bounded functions yield $\tilde{P}^+\1=M_\rho \1$ and hence by equation \eqref{eq1}, 
\[\PP(X_0\ge 0,\dots,X_N\ge 0) = \int_0^\infty M_\rho^N(\1)\di\tilde{\mu},\]
where $\tilde{\mu}=\mathcal{N}(0,1-\rho^2)$.
Note that we related the persistence problem to an $L^2$-operator, contrary to relation \eqref{eq1}. This has two advantages: $M_\rho$ is compact and allows to use perturbation theory, contrary to $P^+$.

The proof of Theorem \ref{PE=Largest Eigenvalue} is based on a Perron-Frobenius statement for integral operators \cite{Jentzsch}. We use the following version of this theorem as stated in 
\cite[Chapter V, Theorem 6.6]{Schaefer}.

\begin{proposition}\label{Jentzsch 1912}
Let $E=L^p(\Omega,\mathcal{F},\nu)$, where $(\Omega,\mathcal{F},\nu)$ is a $\sigma$-finite measure space and $1\le p\le \infty$. Suppose $T\colon E\to E$ is a bounded integral operator given by a measurable kernel $k\ge 0$ satisfying the following two assumptions:
\begin{enumerate}
\item[(a)] some power of $T$ is compact,
\item[(b)] for all $B\in\mathcal{F}$ such that $\nu(B)>0$ and $\nu(B^C)>0$ it holds that
\[\int_{B^C} \int_B k(x,y)\di\nu (x) \di\nu (y) > 0.\]
\end{enumerate}
Then $\lambda(T)$, the spectral radius of $T$, is an eigenvalue of $T$ with a unique normalized, positive eigenfunction $f$, i.e.\ $\|f\|_{L^p(\nu)}=1$ and $f>0$ $\nu$-a.e.\ and any eigenfunction $\tilde{f}$ with these two properties coincides with $f$ $\nu$-a.e.
\end{proposition}

To obtain a relation between the persistence exponent and the largest eigenvalue of some integral operator, we will use the following lemma which may be of interest when generalizing our results to other situations.

\begin{lemma}\label{Lemma Perron-Frobenius}
Under the hypotheses of the above proposition, if moreover $\1\in E$, the eigenfunction $f$ is bounded and $\mu$ is a finite measure with $\di\mu=g\di\nu$, $g\in L^q(\Omega,\mathcal{F},\nu)$ with $\frac{1}{p}+\frac{1}{q}=1$, then
\[\int_\Omega T^N(\1)(x) \di\mu(x) =\lambda(T)^{N+o(N)}.\] 
Additionally, if $p=2$ and the operator $T$ is normal, i.e.\ $TT^\ast = T^\ast T$, where by $T^\ast$ the adjoint operator is denoted, we get
\[c\lambda(T)^N \le \int_\Omega T^N(\1)(x) \di\mu(x) \le C\lambda(T)^N \]
for some constants $c,C>0$.
\end{lemma}

\begin{proof}[Proof of Lemma \ref{Lemma Perron-Frobenius}]
\textit{Upper bound:} We define a functional by $\varphi_\mu\colon L^p(\nu)\to\R, f\mapsto \int_\Omega f(x)\di\mu(x) =\int_\Omega f(x)g(x)\di\nu(x)$. It holds that $\|\varphi_\mu\| = \|g\|_{L^q(\nu)}$. We obtain
\begin{align*}
\int_\Omega T^N(\1)(x)\di\mu(x) = \varphi_\mu(T^N(\1)) &\le \|\varphi_\mu\|\cdot \|T^N(\1)\|_{L^p(\nu)}\\
& \le \|g\|_{L^q(\nu)}\cdot \|T^N\|\cdot \|\1\|_{L^p(\nu)}\\
&= \lambda(T)^{N+o(N)},
\end{align*} 
since $\1\in L^p(\nu), g\in L^q(\nu)$ by assumption and $\lambda(T)=\lim_{N\to\infty} \|T^N\|^{\frac{1}{N}}$.
If $p=2$ and $T$ is normal, then we have $\|T^N\|=\lambda(T^N) = \lambda(T)^N$ due to the spectral mapping theorem (see e.g.\ \cite[Chapter VIII, Theorem 2.7]{Conway}) and thus $\int_\Omega T^N(\1)\di\mu(x)\le C\lambda(T)^N$.\\

\textit{Lower bound:} Since the eigenfunction $f$ is bounded by assumption, we have $\1(x)\ge \frac{f(x)}{\|f\|_{\infty}}$. Hence,
\begin{align*}
\int_\Omega T^N(\1)(x)\di\mu(x) &\ge \int_\Omega T^N\left(\frac{f(x)}{\|f\|_{\infty}}\right)\di\mu(x)\\
&= \lambda(T)^N \int_\Omega \frac{f(x)}{\|f\|_{\infty}}\di\mu(x)\\
&= c\cdot \lambda(T)^N,
\end{align*}
where the first inequality is due to the non-negativity of the kernel $k$.
\end{proof}

\begin{proof}[Proof of Theorem \ref{PE=Largest Eigenvalue}]
The assertion of the theorem is a consequence of Lemma \ref{Lemma Perron-Frobenius} applied to $M_\rho$. Therefore, we need to check that we are in the setting of Lemma \ref{Lemma Perron-Frobenius}.\\
The operator $M_\rho$ is compact and, since $m_\rho(x,y)>0$ for all $x,y$, condition (b) of the above proposition is satisfied. So the spectral radius $\lambda(M_\rho)$ is an eigenvalue of $M_\rho$ with a unique normalized, positive eigenfunction. To obtain the boundedness of this eigenfunction, note that we get an eigenfunction $f\in L^\infty([0,\infty))$ of the operator $\tilde{P}^+$ which is positive with corresponding eigenvalue $\lambda_\rho$ due to \cite[Theorem 2.1 \& Theorem 2.3]{Aurzada}. (In \cite{Aurzada} only a non-negative eigenfunction is obtained, but since in our case the operator $\tilde{P}^+$ is irreducible, an application of the above proposition yields that the eigenfunction is actually positive.) Since $\tilde{P}^+ g=M_\rho g$ for bounded $g$, the positive function $f$ is an eigenfunction of $M_\rho$. Therefore, the corresponding eigenvalue $\lambda_\rho$ is the largest one, i.e. $\lambda_\rho = \lambda(M_\rho)$. To summarize, $\lambda_\rho$ is the largest eigenvalue of $M_\rho$ with a positive and bounded eigenfunction.\\
In addition, we have
\[g(x):= \frac{\di\tilde{\mu}}{\di\gamma}(x) = \frac{1}{\sqrt{1-\rho^2}} e^{-\frac{\rho^2 x^2}{2(1-\rho^2)}}\]

and clearly $g\in L^2(\gamma)$.
Therefore, Lemma \ref{Lemma Perron-Frobenius} yields
\[c_\rho\cdot\lambda_\rho^N \le \PP(X_0\ge 0,\dots , X_N\ge 0) \le C_\rho\cdot\lambda_\rho^N,\]
where $\lambda_\rho$ is the largest eigenvalue of the self-adjoint Hilbert-Schmidt integral operator $M_\rho$.
\end{proof}


\subsection{\texorpdfstring{Proofs of the results of Section \ref{Section Peturbation AR Gaussian}}{Proofs of the results of Section 2.2}}

The proof of Theorem \ref{Largest Eigenvalue holomorphic} is based on methods from perturbation theory. A brief introduction including relevant definitions and theorems can be found in Appendix \ref{appendix}.

\begin{proof}[Proof of Theorem \ref{Largest Eigenvalue holomorphic}]
First, we want to show that $M_\rho$ is holomorphic. Let us consider the operator norm of the integral operator $M^{(n)}$ on $L^2([0,\infty),\gamma)$ with kernel $a_n(x,y):=\frac{1}{n!}h_n(x)h_n(y)$. 
We compute 
\begin{align*}
& \int_0^\infty \int_0^\infty a_n(x,y)^2 \di\gamma(x) \di\gamma(y)\\
&= \int_0^\infty \int_0^\infty \left(\frac{1}{n!} h_n(x) h_n(y)\right)^2 \di\gamma(x) \di\gamma(y)\\
&= \int_0^\infty \frac{1}{n!}h_n(x)^2 \di\gamma(x) \cdot \int_0^\infty \frac{1}{n!}h_n(y)^2 \di\gamma(y)\\
&= \frac{1}{4}.
\end{align*}
This implies 

\begin{equation}\label{eq: inequality operator norm M(n)}
\|M^{(n)}\| \le \frac{1}{2}.
\end{equation}
 
Accordingly, $\tilde{M_\rho}:=\sum_{n\in\N}\rho^nM^{(n)}$ exists on the disc $|\rho|<1$. The holomorphicity of $M_\rho$ will be proved, if we show that $M_\rho=\tilde{M_\rho}$. Let $B:=\{h_n(x)\1_{[0,N]}(x)\colon n,N\in\N\}$. $B$ is a fundamental subset since $(h_n)_{n\in\N}$ is a fundamental subset of $L^2([0,\infty),\gamma)$ (see e.g. \cite[Chapter 2]{Janson}) and $h_n\in\overline{B}$ for every $n\in\N$. Hence it is sufficient to show $M_\rho f=\tilde{M_\rho} f$ for all $f\in B$. 
For $f\in B$ and $x\in[0,\infty)$ we get
\begin{align}
 M_\rho f(x)&=  \int_0^\infty m_\rho(x,y) f(y) \di\gamma(y) \nonumber \\
&= \int_0^\infty \sum_{n=0}^\infty \rho^n a_n(x,y) f(y) \di\gamma(y) \label{eq: M_rho}
\end{align}

By \cite[Equation (4.15)]{Janson}, we have
\[\sum_{n=0}^\infty \left|\rho^n a_n(x,y)\right|\le \EE\left[e^{|\rho|(x+|\eta|)(y+|\zeta|)} \right] =: C(x,y)\] 
where $\eta,\zeta$ are i.i.d.\ random variables with standard normal distribution. Let us denote $f(x)= h_{k}(x)\1_{[0,N](x)}$.
We can exchange sum and integral in \eqref{eq: M_rho} since
\begin{align*}
 &\int_0^\infty \sum_{n=0}^\infty \left|\rho^n a_n(x,y) f(y) \right|\di\gamma(y) \\
&\le  \int_0^\infty C(x,y) \left| f(y) \right|\di\gamma(y) \\
&= \int_0^{N} |h_{k}(y)|C(x,y) \di\gamma(y)\\
&\le C(x,N)\cdot  \int_0^{N} |h_{k}(y)| \di\gamma(y) \\
&< \infty.
\end{align*}

So, we obtain
\[M_\rho f =  \sum_{n=0}^\infty \rho^n \int_0^\infty a_n(x,y) f(y) \di\gamma(y) = \tilde{M_\rho}f,\]
showing that $M_\rho$ is holomorphic.
Furthermore, $\lambda_0$ is an eigenvalue with algebraic multiplicity equal to $1$ \cite[Chapter V, Theorem 5.2]{Schaefer}. 
The assertion of the theorem now follows from Appendix \ref{appendix} Theorem \ref{Thm: Eigenvalue, eigenvector holomorphic} and Corollary \ref{Cor: radius of convergence HR} by using equation \eqref{eq: inequality operator norm M(n)}.

\end{proof}


\noindent\textbf{Acknowledgement.} We would like to thank two anonymous referees for their comments that helped to improve the exposition of this paper.\\
This work was supported by the Deutsche Forschungsgemeinschaft (grant AU370/4).

\begin{appendices}
\section{Collection of auxiliary results from perturbation theory}\label{appendix}
\subsection{Overview}
The aim of this appendix is to give the reader a simple and almost self-contained introduction to the relevant results of perturbation theory for this paper. The appendix is based on the classical work \cite{Kato}.\\
Let $X$ be a Banach space with norm $\|\cdot\|_{X}$. We write $\mathcal{L}(X)$ for the set of bounded linear operators on $X$. For $T\in\mathcal{L}(X)$ we will denote by $\|T\|$ the operator norm of $T$, i.e.\ $\|T\|:=\inf\{c\ge 0\colon \|Tx\|_{X}\le c\|x\|_{X}\mbox{ for all }x\in X \}$.
For a sequence $(T_n)_n$ of operators on $X$ and an operator $T$ on $X$, we write $\lim_{n\to\infty} T_n=T$ if $\lim_{n\to\infty}\|T_n-T\|= 0$. Moreover, let $D\subseteq \C$ be a domain.

\begin{definition}
The operator-valued function $\mathcal{T}\colon D\to \mathcal{L}(X)$ is called holomorphic if $\lim_{h\to 0}\frac{\mathcal{T}(t+h)-\mathcal{T}(t)}{h}$ exists for all $t\in D$. 
\end{definition}

Throughout the appendix, we will make use of properties of holomorphic operator-valued functions. Roughly speaking, the results of complex-valued functions can be generalized to operator-valued functions by considering scalar-valued functions defined via the dual space. For an overview of this topic we refer to \cite[Section 3.3, Section 10.1]{Baumgaertel} and \cite[Appendix A]{Arendt}.

A holomorphic operator-valued function on a disc can be expressed as a power series on this disc. Conversely, a convergent power series on a disc defines a holomorphic function.
To simplify notation we continue to write $T_t$ for $\mathcal{T}(t)$. The key to prove Theorem \ref{Largest Eigenvalue holomorphic} are the following results from perturbation theory.

\begin{theorem}\label{Thm: Eigenvalue, eigenvector holomorphic}
Assume that $\mathcal{T}\colon D\to\mathcal{L}(X)$ is holomorphic. Let $t_0\in D$ and $\lambda_0$ be an isolated eigenvalue of $T_{t_0}$ with algebraic multiplicity equal to one. Then there exists an open neighbourhood $U\subseteq D$ of $t_0$ and a holomorphic function $\lambda\colon U\to \C$ such that $\lambda_t$ is an eigenvalue of $T_t$ for $t\in U$.\\
In addition, there exists an open neighbourhood $U'\subseteq D$ of $t_0$ and a holomorphic function $g\colon U'\to X$ such that $g_t$ is an eigenvector of $T_t$ with eigenvalue $\lambda_t$ for $t\in U'\cap U$.
\end{theorem}  

\begin{theorem}\label{Thm: radius of convergence BR}
Under the conditions of Theorem \ref{Thm: Eigenvalue, eigenvector holomorphic} with $T_t = \sum_{n\in\N}(t-t_0)^n T^{(n)}$ assume that $\|T^{(n)}\|\le ac^{n-1}$ for all $n\in\N$ with $a,c\ge 0$. The following lower bound for the radius of convergence $r$ of the power series of $\lambda_t$ at $t_0$ holds:
\[r\ge \min_{\lambda\in\Gamma}\frac{1}{a\|R_0(\lambda)\|+c},\]
where $R_0(\cdot)$ is the resolvent operator of $T_{t_0}$ and $\Gamma$ is an arbitrary closed curve that lies outside $\Sigma(T_{t_0})$ with positive direction which encloses $\lambda_0$, where $\Sigma(T_{t_0})$ is the spectrum of $T_{t_0}$.
\end{theorem}

\begin{corollary}\label{Cor: radius of convergence HR}
Under the assumptions of Theorem \ref{Thm: radius of convergence BR}, and if moreover $X$ is a Hilbert space and $T_{t_0}$ is normal, we get
\[r\ge \frac{1}{\frac{2a}{d}+c},\]
where $d:=\dist\left(\lambda_0,\Sigma(T_{t_0})\setminus\{\lambda_0\}\right)$.
\end{corollary}


\subsection{Preliminaries}

\begin{lemma}\label{Lemma: Neumann series}
Let $T\in\mathcal{L}(X)$ with $\|T\|<1$. Then the Neumann series $\sum_{n=0}^\infty T^n$ converges in the operator norm and we have
\[(\Id-T)^{-1} = \sum_{n=0}^\infty T^n.\]
\end{lemma}
This is a well-known result in functional analysis and can be found for instance in \cite[Propostion I.1.6]{Takesaki}.

\begin{definition}
For $T\in\mathcal{L}(X)$ the resolvent set $\sigma(T)$ is the set of all $\lambda\in\C$ such that $(T-\lambda\Id)$ has a bounded inverse.\\
The operator-valued function $R\colon \sigma(T)\to \mathcal{L}(X),\ \lambda\mapsto (T-\lambda\Id)^{-1}$ is called the resolvent operator of $T$.
\end{definition}
For abbreviation, we write $T-\lambda$ instead of $T-\lambda\Id$ when no confusion can arise.

\begin{lemma}\label{Lemma: resolvent equation}
We have
\[R(\lambda')-R(\lambda) = (\lambda'-\lambda)R(\lambda')R(\lambda)\]
for all $\lambda,\lambda'\in\sigma(T)$. In particular, $R(\lambda)$ and $R(\lambda')$ commute.
\end{lemma}

\begin{proof}
The following easy computation shows the statement:
\begin{align*}
R(\lambda')-R(\lambda) &= R(\lambda')(T-\lambda)R(\lambda)-R(\lambda')(T-\lambda')R(\lambda)\\
&= -R(\lambda')\lambda R(\lambda) + R(\lambda')\lambda' R(\lambda) \\
&= (\lambda'-\lambda)R(\lambda')R(\lambda).
\end{align*}

\end{proof}

\begin{proposition}\label{Neumann series for the resolvent}
Let $\lambda,\lambda_0\in \sigma(T)$ with $|\lambda-\lambda_0|<\|R(\lambda_0)\|^{-1}$ then the so-called first Neumann series for the resolvent $\sum_{n=0}^\infty (\lambda-\lambda_0)^n R(\lambda_0)^{n+1}$ is convergent. In this case, we have
\[R(\lambda)= \sum_{n=0}^\infty (\lambda-\lambda_0)^n R(\lambda_0)^{n+1}.\]
\end{proposition}

This shows that $R(\cdot)$ is holomorphic on $\sigma(T)$.

\begin{proof}
Let $\lambda,\lambda_0\in\sigma(T)$. By Lemma \ref{Lemma: resolvent equation} we obtain 
\[R(\lambda) = R(\lambda_0)\left(1-(\lambda-\lambda_0)R(\lambda_0) \right)^{-1}.\]
If $\|(\lambda-\lambda_0)R(\lambda_0)\|<1$, Lemma \ref{Lemma: Neumann series} implies that 
\[R(\lambda) = R(\lambda_0)\sum_{n=0}^\infty (\lambda-\lambda_0)^n R(\lambda_0)^n,\]
which proves the statement.
\end{proof}


\subsection{Perturbation of the resolvent operator}

We define
\[R(t,\lambda)= (T_t-\lambda)^{-1}\]
for any $(t,\lambda)$ with $\lambda\in\sigma(T_t)$. We already know from the last proposition that $R(t,\lambda)$ is holomorphic in $\lambda$ for each fixed $t$. Now we will show that $R(t,\lambda)$ is holomorphic in both variables.

\begin{proposition}\label{Prop: resolvent holomorphic}
Let $\mathcal{T}\colon D\to\mathcal{L}(X)$ be holomorphic. Then $R(t,\lambda)$ is holomorphic in both variables $t$ and $\lambda$.
\end{proposition}

\begin{proof}
Let $(t_0,\lambda_0)$ such that $\lambda_0\in\sigma(T_{t_0})$. Since $\mathcal{T}$ is holomorphic we can write $T_t = \sum_{n=0}^\infty T^{(n)}(t-t_0)^n$ for $|t-t_0|$ small. We have
\begin{align*}
T_t-\lambda &= T_{t_0}-\lambda_0 - (\lambda-\lambda_0)+\sum_{n=1}^\infty T^{(n)}(t-t_0)^n\\
&= (T_{t_0}-\lambda_0)\Big(1-\big(\lambda-\lambda_0-\sum_{n=1}^\infty T^{(n)}(t-t_0)^n\big)R(t_0,\lambda_0) \Big).
\end{align*}  
Thus,
\[R(t,\lambda) = (T_t-\lambda)^{-1} = R(t_0,\lambda_0)\Big(1-\big(\lambda-\lambda_0-\sum_{n=1}^\infty T^{(n)}(t-t_0)^n\big)R(t_0,\lambda_0) \Big)^{-1}.\]
If
\begin{equation}\label{condition for resolvent (power series)}
\|\big(\lambda-\lambda_0-\sum_{n=1}^\infty T^{(n)}(t-t_0)^n\big)R(t_0,\lambda_0)\|<1,
\end{equation}
the last expression can be written as a double series in $\lambda$ and $t$ by Lemma \nolinebreak \ref{Lemma: Neumann series}. This condition is satisfied if 
\begin{equation}\label{condition for resolvent (power series) II}
|\lambda-\lambda_0|+\sum_{n=1}^\infty |t-t_0|^n\|T^{(n)}\| < \|R(t_0,\lambda_0)\|^{-1},
\end{equation}

which is the case if $|\lambda-\lambda_0|$ and $|t-t_0|$ are small enough.
\end{proof}

\begin{remark}
If we fix $\lambda$ and write $R(t,\lambda)$ as a power series in $t$ at $t_0$ we get
\[R(t,\lambda) = R(t_0,\lambda)\Big(\sum_{k=0}^\infty\big(-\sum_{n=1}^\infty T^{(n)}(t-t_0)^n R(t_0,\lambda)\big)^k\Big).\]
We can rewrite this as
\begin{equation}\label{eq. second Neumann}
R(t,\lambda) = \sum_{k=0}^\infty (t-t_0)^k R^{(k)}(\lambda),
\end{equation}
with 
\[R^{(k)}(\lambda) = \sum_{\substack{k_1+\cdots+k_n=k\\k_i\ge 1,\,n\ge 1}} (-1)^n R(t_0,\lambda)T^{(k_1)}R(t_0,\lambda)T^{(k_2)}\cdots T^{(k_n)}R(t_0,\lambda).\]
The right-hand side of equation \eqref{eq. second Neumann} is called the \textit{second Neumann series for the resolvent}. 
\end{remark}


\subsection{Perturbation of the eigenprojection, eigenvalue, and eigenvector}

\begin{proposition}\label{eigenprojection}
Let $T\in\mathcal{L}(X)$ and $\lambda_0$ be a simple isolated eigenvalue of $T_0$. Let $\Gamma$ be a closed curve in $\sigma(T)$ with positive direction which encloses $\lambda_0$. Then
\[P_0 := - \frac{1}{2\pi i}\int_\Gamma R(\lambda)\di\lambda\]
is a projection on the eigenspace of $\lambda_0$.
\end{proposition}

\begin{proof}
We need to show that: \begin{enumerate}
\item[(i)] $P_0$ is a projection, i.e.\ $P_0^2=P_0$.
\item[(ii)] $\mathbf{R}(P_0)= M_0$, where $\mathbf{R}(P_0)$ is the range of $P_0$ and $M_0$ is the eigenspace of $\lambda_0$.
\end{enumerate}
Let $\Gamma'$ be a closed curve in $\sigma(T)$ with positive direction which encloses $\lambda_0$ and lies outside $\Gamma$.
Then $\int_\Gamma R(\lambda)\di\lambda = \int_{\Gamma'} R(\lambda)\di\lambda$ and we have

\begin{align*}
(-2\pi i)^2 P_0^2 &= \int_\Gamma R(\lambda)\di\lambda\cdot \int_{\Gamma'} R(\mu)\di\mu \\
&= \int_{\Gamma'} \int_\Gamma R(\lambda) R(\mu) \di\lambda\di\mu\\
&= \int_{\Gamma'} \int_\Gamma \frac{1}{\mu-\lambda} R(\mu)\di\lambda\di\mu - \int_\Gamma \int_{\Gamma'} \frac{1}{\mu-\lambda} R(\lambda)\di\mu\di\lambda\\
&= \int_{\Gamma'} 0 \di\mu - \int_\Gamma 2\pi i R(\lambda)\di\lambda\\
&= -2\pi i \int_\Gamma R(\lambda)\di\lambda\\
&= (-2\pi i)^2 P_0,
\end{align*}

where the third equality follows by Lemma \ref{Lemma: resolvent equation}.
This shows (i). To prove (ii), we begin by showing $M_0\subseteq\mathbf{R}(P_0)$.
Let $x\in M_0$, i.e.\ $T(x)=\lambda_0 x$.  Then
\[P_0(x)= -\frac{1}{2\pi i}\int_\Gamma (T-\lambda\Id)^{-1}(x)\di\lambda = -\frac{1}{2\pi i}\int_\Gamma (\lambda_0-\lambda)^{-1}(x)\di\lambda = x.\]
Now, we proceed to show that $\mathbf{R}(P_0)\subseteq M_0$. We compute
\begin{align*}
(-2\pi i)TP_0 &= \int_\Gamma T(T-\lambda)^{-1}\di\lambda\\
&= \int_\Gamma \Id+\lambda(T-\lambda)^{-1}\di\lambda\\
&= \int_\Gamma \lambda(T-\lambda)^{-1}\di\lambda\\
&= \Res_{\lambda_0}(\lambda(T-\lambda)^{-1})\\
&= \lambda_0 \Res_{\lambda_0}((T-\lambda)^{-1})\\
&= \lambda_0 \int_\Gamma R(\lambda)\di\lambda\\
&= (-2\pi i)\lambda_0 P_0.
\end{align*}
Hence, for all $x\in X$ we get $P_0(x)\in M_0$, which completes the proof.
\end{proof}

In what follows, we assume that $\mathcal{T}\colon D\to\mathcal{L}(X)$ is holomorphic and that $\lambda_0$ is an isolated eigenvalue of $T_{t_0}$, $t_0\in D$, with algebraic multiplicity equal to one.
Let $\Gamma$ be a closed curve in $\sigma(T_{t_0})$ with positive direction which encloses $\lambda_0$.

\begin{proposition}\label{Prop: projection holomorphic}
The operator-valued function $\mathcal{P}\colon D\to \mathcal{L}(X),\quad t\to P_t:=\mathcal{P}(t)=\int_\Gamma R(t,\lambda)\di\lambda$ is holomorphic at an open neighbourhood of $t_0.$ It holds that $\dim (P_0X)=\dim (P_tX)$.
\end{proposition}

\begin{proof}
Since $\inf_{\lambda\in\Gamma}\|R(t_0,\lambda)\|^{-1}>0$, from \eqref{condition for resolvent (power series) II} it follows that the second Neumann series for the resolvent is uniformly convergent for $\lambda$ if $|t-t_0|$ is sufficiently small.
In particular, $R(\lambda,t)$ is well-defined for such $(\lambda,t)$ and thus, $\Gamma\subseteq \sigma(T_t)$. Hence, $P_t$ is well-defined for $|t-t_0|$ small and due to the proposition in \ref{Prop: resolvent holomorphic} we get that $P_t$ is holomorphic at an open neighbourhood of $t_0$.
The equality of the dimensions follows by \cite[Lemma I.4.10]{Kato}.
\end{proof}

\begin{proof}[Proof of Theorem \ref{Thm: Eigenvalue, eigenvector holomorphic}]
Combining the last two propositions we see that $P_t$ is the eigenprojection for $T_t$ on the eigenspace of a simple eigenvalue $\lambda_t$ and that $P_t$ is holomorphic in $t$.
Accordingly, we deduce a perturbation series for the eigenvalue $\lambda_t$ via the formula $\lambda_t=\tr(T_tP_t)$.\\
Let $g_{0}$ be an eigenfunction of $\lambda_{0}$. Then $g_t= P_tg_{0}\in P_tX$ and thus an element of the eigenspace of $\lambda_t$. If $g_t\neq 0$ it is an eigenfunction of $\lambda_t$. $P_t$ is holomorphic and hence $g_t$ is holomorphic. Since $g_{0}\neq 0$, we have that $g_t\neq 0$ at least for $|t-t_0|$ small.
\end{proof}


\subsection{Lower bound for the radius of convergence}
Let $T_t=\sum_{n=0}^\infty (t-t_0)^n T^{(n)}$ for $t\in D$ and $\lambda_0$ be an isolated eigenvalue of $T_{t_0}$ with algebraic multiplicity equal to one. 
It follows from equation \eqref{condition for resolvent (power series)} that for a fixed $\lambda$ the power series $\sum_{n=0}^\infty R^{(n)}(\lambda)(t-t_0)^n$ is convergent, if $\|R(t_0,\lambda)\sum_{n=1}^\infty T^{(n)}(t-t_0)^n\|<1$. 

\begin{proof}[Proof of Theorem \ref{Thm: radius of convergence BR}]
Assume that $\|T^{(n)}\|\le ac^{n-1}$ with $a,c\ge 0$. Then the power series is convergent if 
\[|t-t_0|< \frac{1}{\|R(t_0,\lambda)\|\cdot a+c}.\]
Consequently, $P_t=\int_\Gamma R(t,\lambda) \di\lambda$ can be expressed as a power series if $|t-t_0|< \inf_{\lambda\in\Gamma}\frac{1}{\|R(t_0,\lambda)\|\cdot a +c}$.
\end{proof}

The so obtained lower bound of the radius of convergence of $P_t$, and therewith of $\lambda_t$, depends crucially on the chosen curve $\Gamma$. It is worthwhile to get this bound as large as possible.\\

\begin{proof}[Proof of Corollary \ref{Cor: radius of convergence HR}]
If $T_{t_0}$ is a normal operator on a Hilbert space we have that $R(t_0,\lambda)$ is normal and as a consequence the operator norm of $R(t_0,\lambda)$ is equal to the spectral radius of $R(t_0,\lambda)$.
From this we conclude
\[\|R(t_0,\lambda)\| = \dist(\lambda,\Sigma(T_{t_0}))^{-1},\]
where $\Sigma(T_{t_0})$ is the spectrum of $T_{t_0}$ and $\dist$ is the Hausdorff distance. Let $d:=\dist(\lambda_0,\Sigma(T_{t_0})\setminus \{\lambda_0\})$ and let $\Gamma$ be a circle with radius $\frac{d}{2}$ and center $\lambda_0$. Then $\|R(t_0,\lambda)\|=\frac{2}{d}$ for every $\lambda\in\Gamma$ and we get
\[r\ge \frac{1}{\frac{2a}{d}+c}.\]
\end{proof}

\end{appendices}


\begin{thebibliography}{10}

\bibitem{Arendt}
W.~Arendt, C.~J. Batty, M.~Hieber, and F.~Neubrander.
\newblock {\em Vector-valued {L}aplace transforms and {C}auchy problems},
  volume~96 of {\em Monographs in Mathematics}.
\newblock Birkh\"{a}user/Springer Basel AG, Basel, 2011.

\bibitem{Baumgarten2}
F.~Aurzada and C.~Baumgarten.
\newblock Survival probabilities of weighted random walks.
\newblock {\em ALEA, Lat.\ Am.\ J.\ Probab.\ Math. Stat.}, 8:235--258, 2011.

\bibitem{Aurzada}
F.~Aurzada, S.~Mukherjee, and O.~Zeitouni.
\newblock Persistence exponents in {M}arkov chains.
\newblock {\em arXiv preprint arXiv:1703.06447}, 2017.

\bibitem{aurzadasimon}
F.~Aurzada and T.~Simon.
\newblock Persistence probabilities and exponents.
\newblock In {\em L\'evy matters. {V}}, volume 2149 of {\em Lecture Notes in
  Math.}, pages 183--224. Springer, Cham, 2015.

\bibitem{Baumgaertel}
H.~Baumg{\"a}rtel.
\newblock {\em Analytic perturbation theory for matrices and operators},
  volume~15 of {\em Operator Theory: Advances and Applications}.
\newblock Birkh\"{a}user Verlag, Basel, 1985.

\bibitem{Baumgarten}
C.~Baumgarten.
\newblock Survival probabilities of autoregressive processes.
\newblock {\em ESAIM: Probability and Statistics}, 18:145--170, 2014.

\bibitem{braymajumdarschehr}
A.~Bray, S.~N. Majumdar, and G.~Schehr.
\newblock Persistence and first-passage properties in non-equilibrium systems.
\newblock {\em Advances in Physics}, 62(3):225--361, 2013.

\bibitem{Champagnat}
N.~Champagnat and D.~Villemonais.
\newblock General criteria for the study of quasi-stationarity.
\newblock {\em arXiv preprint arXiv:1712.08092}, 2017.

\bibitem{Collet}
P.~Collet, S.~Mart{\'\i}nez, and J.~San~Mart{\'\i}n.
\newblock {\em Quasi-stationary distributions: Markov chains, diffusions and
  dynamical systems}.
\newblock Springer Science \& Business Media, 2012.

\bibitem{Conway}
J.~B. Conway.
\newblock {\em A course in functional analysis}, volume~96 of {\em Graduate
  Texts in Mathematics}.
\newblock Springer Science \& Business Media, 2007.

\bibitem{Dembo}
A.~Dembo, J.~Ding, and J.~Yan.
\newblock Persistence versus stability for auto-regressive processes.
\newblock {\em arXiv preprint arXiv:1906.00473}, 2019.

\bibitem{Wachtel}
G.~Hinrichs, M.~Kolb, and V.~Wachtel.
\newblock Persistence of one-dimensional {AR}(1)-sequences.
\newblock {\em arXiv preprint arXiv:1801.04485}, 2018.

\bibitem{Janson}
S.~Janson.
\newblock {\em Gaussian {H}ilbert spaces}, volume 129 of {\em Cambridge Tracts
  in Mathematics}.
\newblock Cambridge University Press, Cambridge, 1997.

\bibitem{Jentzsch}
R.~Jentzsch.
\newblock {\"U}ber {I}ntegralgleichungen mit positivem {K}ern.
\newblock {\em Journal f{\"u}r die reine und angewandte Mathematik},
  141:235--244, 1912.

\bibitem{Kato}
T.~Kato.
\newblock {\em Perturbation theory for linear operators}, volume 132 of {\em
  Die {G}rundlehren der mathematischen {W}issenschaften}.
\newblock Springer-Verlag New York, Inc., New York, 1966.

\bibitem{Letac}
G.~Letac.
\newblock Isotropy and sphericity: Some characterisations of the normal
  distribution.
\newblock {\em The Annals of Statistics}, 9(2):408--417, 1981.

\bibitem{Majumdarbray}
S.~N. Majumdar, A.~J. Bray, and G.~C. Ehrhardt.
\newblock Persistence of a continuous stochastic process with discrete-time
  sampling.
\newblock {\em Physical Review E}, 64, 015101(R), 2001.

\bibitem{Majumdardhar}
S.~N. Majumdar and D.~Dhar.
\newblock Persistence in a stationary time series.
\newblock {\em Physical Review E}, 64, 046123, 2001.

\bibitem{Mehler}
F.~G. Mehler.
\newblock Ueber die {E}ntwicklung einer {F}unction von beliebig vielen
  {V}ariablen nach {L}aplaceschen {F}unctionen h\"oherer {O}rdnung.
\newblock {\em Journal f\"ur die reine und angewandte Mathematik}, 66:161--176,
  1866.

\bibitem{Meleard}
S.~M{\'e}l{\'e}ard and D.~Villemonais.
\newblock Quasi-stationary distributions and population processes.
\newblock {\em Probability Surveys}, 9:340--410, 2012.

\bibitem{metzleretal}
R.~Metzler, G.~Oshanin, and S.~Redner.
\newblock {\em First-passage Phenomena and their applications}.
\newblock World Scientific, 2014.

\bibitem{redner}
S.~Redner.
\newblock {\em A {G}uide to {F}irst-{P}assage {P}rocesses}.
\newblock Cambridge University Press, Cambridge, 2001.

\bibitem{Schaefer}
H.~H. Schaefer.
\newblock {\em Banach Lattices and Positive Operators}.
\newblock Springer, Berlin, Heidelberg, 1974.

\bibitem{Takesaki}
M.~Takesaki.
\newblock {\em Theory of operator algebras I}, volume 124 of {\em Encyclopaedia
  of Mathematical Sciences}.
\newblock Springer-Verlag, Berlin, 2002.

\bibitem{Tweedie2}
R.~L. Tweedie.
\newblock Quasi-stationary distributions for {M}arkov chains on a general state
  space.
\newblock {\em Journal of Applied Probability}, 11(4):726--741, 1974.

\bibitem{Tweedie1}
R.~L. Tweedie.
\newblock {R}-theory for {M}arkov chains on a general state space {I}:
  solidarity properties and {R}-recurrent chains.
\newblock {\em The Annals of Probability}, 2(5):840--864, 1974.

\end{thebibliography}
\end{document}